\documentclass[12pt]{amsart}
\usepackage{amsmath}
\usepackage{amsthm}
\usepackage{amssymb}
\usepackage{amscd}

\usepackage{comment}

\usepackage{hyperref}
\usepackage{tikz}
\usetikzlibrary{matrix,arrows}

\DeclareMathOperator{\lc}{H}

\newcommand{\hght}{\operatorname{ht}}
\newcommand{\Spec}{\operatorname{Spec}}
\newcommand{\length}{\ell}

\newcommand{\eh}{\operatorname{e}}

\newcommand{\mf}{\mathfrak}

\DeclareMathOperator{\m}{\mathfrak{m}}

\DeclareMathOperator{\initial}{in}

\DeclareMathOperator{\red}{red}

\newcommand{\gr}{\operatorname{gr}}

\newtheorem{theorem}{Theorem}
\newtheorem{lemma}[theorem]{Lemma}

\newtheorem{corollary}[theorem]{Corollary}

\newtheorem{conjecture}[theorem]{Conjecture}
\newtheorem*{statement*}{Statement}
\newtheorem*{theorem*}{Theorem}
\newtheorem*{lemma*}{Lemma}
\newtheorem*{fact*}{Fact}

\theoremstyle{definition}
\newtheorem{definition}[theorem]{Definition}
\newtheorem*{definition*}{Definition}

\newtheorem*{example*}{Example}

\newtheorem{remark}[theorem]{Remark}

\begin{document}

\title{Uniform Lech's inequality}

\author{Linquan Ma}
\address{Department of Mathematics, Purdue University, West Lafayette, IN 47907 USA}
\email{ma326@purdue.edu}

\author{Ilya Smirnov}
\address{BCAM -- Basque Center for Applied Mathematics, Bilbao, Spain \quad and \quad IKERBASQUE, Basque Foundation for Science, Bilbao, Spain}
\email{ismirnov@bcamath.org}

\maketitle

\begin{abstract}
Let $(R,\m)$ be a Noetherian local ring of dimension $d\geq 2$. We prove that if $\eh(\widehat{R}_{\red})>1$, then the classical Lech's inequality can be improved uniformly for all $\m$-primary ideals, that is, there exists $\varepsilon>0$ such that $\eh(I)\leq d!(\eh(R)-\varepsilon)\length(R/I)$ for all $\m$-primary ideals $I\subseteq R$. This answers a question raised in \cite{HMQS}. We also obtain partial results towards improvements of Lech's inequality when we fix the number of generators of $I$.
\end{abstract}

\section{Introduction}
The origin of this paper is a simple inequality of Lech, proved in \cite{LechMultiplicity}, that connects the colength and the multiplicity of an $\m$-primary ideal in a Noetherian local ring $(R,\m)$.

\begin{theorem}[Lech's inequality]
\label{theorem: Lech}
Let $(R,\m)$ be a Noetherian local ring of dimension $d$ and let $I$ be an $\m$-primary ideal of $R$. Then we have $$\eh(I)\leq d!\eh(R)\length(R/I),$$ where $\eh(I)$ denotes the Hilbert--Samuel multiplicity of $I$ and $\eh(R):=\eh(\m)$.
\end{theorem}

Lech observed that his inequality is never sharp if $d \geq 2$ (see \cite[Page 74, after (4.1)]{LechMultiplicity}): that is, when $d\geq 2$ we always have a strict inequality in Theorem~\ref{theorem: Lech}. The problem of improving Lech's inequality by replacing $\eh(R)$ with a smaller constant was raised in \cite{HMQS}. This problem is partially motivated by \cite{Mumford}, where Mumford considered the quantity
$$\sup_{\sqrt{I} = \m} \left \{\frac{\eh(I)}{d!\length(R/I)} \right \}$$
and showed that this has close connections with singularities on the compactification of the moduli spaces of smooth varieties constructed via Geometric Invariant Theory. The following conjecture is the proposed refinement of Lech's inequality (see \cite[Conjecture 1.2]{HMQS}):

\begin{conjecture}
\label{conj: asymptoticLech}
Let $(R,\m)$ be a Noetherian local ring of dimension $d\geq 1$.
\begin{enumerate}
\item[(a)] If $\widehat{R}$ has an isolated singularity, i.e.,  $\widehat{R}_P$ is regular for all $P\in\Spec\widehat{R}-\{\m\}$, then
\[
\lim_{N\to\infty} \sup_{\substack{\sqrt{I}=\m \\ \length(R/I)> N}} \left\{\frac{\eh(I)}{d!\length(R/I)} \right\}=1.
\]
\item[(b)] We have $\eh(\widehat{R}_{\red}) > 1$ if and only if
\[
\lim_{N\to\infty} \sup_{\substack{\sqrt{I}=\m \\ \length(R/I)> N}} \left\{\frac{\eh(I)}{d!\length(R/I)} \right\}<\eh(R).
\]
\end{enumerate}
\end{conjecture}

Roughly speaking, we expect that the constant $\eh(R)$ on the right hand side of Lech's inequality can usually be replaced by a smaller number as long as the colength of the ideal is large. The first part of Conjecture~\ref{conj: asymptoticLech}
was established in \cite{HMQS}
when $R$ has positive characteristic with perfect residue field and the second part of Conjecture~\ref{conj: asymptoticLech} when $R$ has equal characteristic. Our main goal in this article is to settle the second part of Conjecture~\ref{conj: asymptoticLech} in full generality by proving the following, which can be viewed as a uniform version of Lech's inequality:

\begin{theorem}[=Theorem~\ref{theorem: uniform Lech}]
\label{theorem: uniform Lech introduction}
Let $(R, \mf m)$ be a Noetherian local ring of dimension $d \geq 2$. Suppose $\eh({\widehat{R}}_{\red})>1$. Then there exists $\varepsilon > 0$ such that for any $\m$-primary ideal $I$, we have
$$
\eh(I) \leq d!(\eh(R) - \varepsilon) \length (R/I).
$$

\end{theorem}

The main case of Conjecture~\ref{conj: asymptoticLech} (b) follows immediately from Theorem~\ref{theorem: uniform Lech introduction}, see Corollary~\ref{cor: asymptotic Lech}. Our approach to Theorem~\ref{theorem: uniform Lech introduction} is similar to the strategy in the equal characteristic case proved in \cite[Theorem 5.8]{HMQS}. However, the main reason that the argument in \cite{HMQS} does not carry to mixed characteristic is because it crucially relies on a refined version of Lech's inequality for ideals with a fixed number of generators in equal characteristic (see \cite[Proposition 5.7]{HMQS}, recalled in Theorem~\ref{theorem: Hanes}) which essentially follows from work of Hanes \cite[Theorem 2.4]{Hanes} on Hilbert--Kunz multiplicity. We do not know whether such a version of Lech's inequality holds in mixed charactersitic (though we expect it to hold, see Conjecture~\ref{conj: Hanes}). Due to the absence of this ingredient in mixed characteristic, we prove Theorem~\ref{theorem: uniform Lech introduction} by carefully passing to certain associated graded rings to reduce to an equal characteristic setting so that \cite[Proposition 5.7]{HMQS} can be applied.

On the other hand, our strategy in the proof of Theorem~\ref{theorem: uniform Lech introduction} does allow us to obtain a weaker version of \cite[Proposition 5.7]{HMQS} valid in all characteristics for integrally closed ideals. A value of this result is not just in mixed characteristic, it also removes the need of
a reduction modulo $p$ argument
used in characteristic $0$ to deduce \cite[Theorem 5.8]{HMQS} from the result of Hanes.

\begin{theorem}[=Corollary~\ref{cor: weak Hanes}]
Let $d \geq 2$ and $N \geq d$ be two positive integers. Then there exists a constant $c = c(N, d) \in (0, 1)$ such that for any Noetherian local ring $(R, \m)$ of dimension $d$ and any $\m$-primary integrally closed ideal $I$ which can be generated by $N$ elements we have
\[
\eh(I) \leq d! c \eh(R) \length (R/I).
\]
\end{theorem}

\subsection*{Acknowledgement} We would like Bernd Ulrich for suggesting the proof of Lemma~\ref{lemma: Ulrich} and the anonymous referee for a greatly improved bound in Lemma~\ref{lemma: bound generators}. We would also like to thank Craig Huneke and Pham Hung Quy for the joint project \cite{HMQS}. The first author is partially supported by NSF Grant DMS \#1901672, NSF FRG Grant \#1952366, and a fellowship from the Sloan Foundation.
The second author had support of a fellowship from ``la Caixa'' Foundation (ID 100010434) and from the European Union’s Horizon 2020 research and innovation programme under the Marie Skłodowska-Curie grant agreement No 847648. The fellowship code is ``LCF/BQ/PI21/11830033''.


\section{Preliminaries}

Throughout this article, all rings are commutative, Noetherian, with multiplicative identity $1$. We use $\length (M)$ to denote the length of a finite $R$-module $M$ and $\mu(M)$ to denote the minimal number of generators of $M$.

\begin{definition}\label{def HS}
Let $(R, \mf m)$ be a Noetherian local ring of dimension $d$ and $I$ be an $\mf m$-primary ideal.
The {\it Hilbert--Samuel multiplicity} of $I$ is defined as
\[
\eh(I) = \lim_{n\to\infty}\frac{d!\length(R/I^n)}{n^d}.
\]
\end{definition}

It is well-known that $\eh(I)$ is always a positive integer. The Hilbert--Samuel multiplicity is closely related to integral closure. Recall that an element $x\in R$ is integral over an ideal $I$ if it satisfies an equation of the form $x^n + a_{1}x^{n-1} + \cdots + a_{n-1}x+ a_n=0$ where $a_k \in I^k$. The set of all elements $x$ integral over $I$ is an ideal and is denoted by $\overline{I}$, called the integral closure of $I$. The Hilbert--Samuel multiplicity is an invariant of the integral closure, i.e., $\eh(I) = \eh(\overline{I})$. Thus, we always have an inequality $\eh(I)/\length(R/I) \leq \eh(\overline{I})/\length(R/\overline{I})$.
In particular, Conjecture~\ref{conj: asymptoticLech} can be restricted to integrally closed ideals. Another related concept is $\m$-full ideals. We briefly recall the definition following \cite{JWatanabe}. 

\begin{definition}\label{def m-full}
Let $(R, \mf m)$ be a Noetherian local ring. Let $\m=(x_1,\dots,x_n)$ and let $\widetilde{R} = R(t_1,\dots,t_n)$, and consider the general linear form $z = t_1x_1+\cdots+t_nx_n$. An ideal $I$ of $R$ is called $\mf m$-{\it full} if $\mf mI\widetilde{R} : z = I\widetilde{R}$.
\end{definition}

The following remark summarizes some useful properties of $\m$-full ideals 

\begin{remark}With notation as in Definition~\ref{def m-full}, we have
\label{remark: property m-full}
\begin{enumerate}
\item If $I$ is $\mf m$-full, then $I\widetilde{R}:z=I\widetilde{R}:\mf m$ (\cite[Lemma~1]{JWatanabe}).
\item If  $I$ is integrally closed, then $I$ is $\mf m$-full  or $I = \sqrt{(0)}$ (\cite[Theorem~2.4]{Goto}).
\item If $I$ is $\mf m$-primary and $\mf m$-full, then $\mu(I) \geq \mu(J)$ for any ideal $J\supseteq I$ (\cite[Theorem~3]{JWatanabe}).
\item If $I$ is $\m$-primary and $\mf m$-full, then $\mu(I) = \length(\widetilde{R}/(z, I)\widetilde{R}) +
\mu(I(\widetilde{R}/z\widetilde{R}))$ (\cite[Theorem~2]{JWatanabe}).
\end{enumerate}
\end{remark}

\subsection* {The associated graded ring}
Our key argument relies on passage to certain associated graded rings
in order to transfer to the equal characteristic setting. We record some notations and simple facts about initial (form) ideals in associated graded rings. Let $J\subseteq R$ be an ideal and let $\gr_J(R)=\bigoplus_n J^n/J^{n+1}$ be the associated graded ring of $R$ with respect to $J$. If $I\subseteq R$ is another ideal then we will use
$$\initial_J(I):= \bigoplus_n \frac{I\cap J^n + J^{n + 1}}{J^{n + 1}}\subseteq \gr_J(R)$$
to denote the initial ideal of $I$ (or form ideal in the notation of \cite{LechMultiplicity}) in the associated graded ring. Now let $(R,\m)$ be a Noetherian local ring and $I\subseteq R$ be an $\m$-primary ideal. It is well-known and easy to check that $\length(\gr_J(R)/\initial_J(I)) = \length(R/I)$. Furthermore, since $\initial_J(I)^n\subseteq \initial_J(I^n)$ and $\dim(R)=\dim(\gr_J(R))$, we have $\eh(I) \leq \eh(\initial_J(I))$.


\begin{lemma}\label{lemma: double graded}
Let $(R, \mf m)$ be a Noetherian local ring.
Then $\gr_{(\mf m, T)} (R[T]) = \gr_{\mf m} (R)[T]$.
Moreover, via this identification,
the initial ideal of a $T$-homogeneous ideal $I = \sum_k I_k T^k$ is
$\sum_k  \initial_{\mf m} (I_k) T^k$. 
\end{lemma}
\begin{proof}
The first claim follows from the second by considering the unit ideal (so that $I_k=R$ for all $k$).
We know that the image of $I$ on the left hand side is
\[
\initial_{(\mf m, T)} (I) = \bigoplus_{n \geq 0} \frac{I \cap (\mf m, T)^n}{I \cap (\mf m, T)^{n + 1}}.
\]
Since $I \cap (\mf m, T)^n = \sum_{k = 0}^{n-1} (I_k \cap \mf m^{n - k}) T^k + \sum_{k \geq n} I_kT^k$,
by restricting to fixed $T$-degree components
we may further decompose
\[
\initial_{(\mf m, T)} (I) = \bigoplus_{n \geq 0} \bigoplus_{k = 0}^n
\frac{I_k \cap \mf m^{n - k}}{I_k \cap \mf m^{n +1 - k}} T^k
=
\bigoplus_{k \geq 0} \left( \bigoplus_{n \geq k}
\frac{I_k \cap \mf m^{n - k}}{I_k \cap \mf m^{n +1 - k}} \right ) T^k
= \bigoplus_{k \geq 0}  \initial_{\mf m} (I_k) T^k.
\qedhere \]
\end{proof}

\section{Main Results}

In this section we prove our main results. We begin with a few lemmas.

\begin{lemma}\label{lemma: bound generators}
Let $(R, \mf m)$ be a Noetherian local ring and $I = I_0 + I_1 T + I_2T^2 + \cdots $ be a $T$-homogeneous
ideal of finite colength in $R[T]$. Then $\mu(I) \leq \mu(I_0) + \length (R/I_0)$.
In particular, if
$\dim(R)=1$, then $\mu(I)\leq \length (R/I_0) + \eh(R) + \length (\lc_{\mf m}^0 (R))$.
\end{lemma}
\begin{proof}
We have containments $I_0 \subseteq I_1 \subseteq \cdots$ and this sequence
will eventually include the unit ideal $R$.
If $x_{k,1}, \ldots, x_{k, D_k}$ for $k \geq 0$ are such that
their images form a minimal generating set for $I_{k+1}/I_k$, then
it is easy to see that $I$ can be generated by
the generators of $I_0$ and $\{x_{k, i}T^{n_k}, T^{n_C}\}_{k = 0, i = 1}^{k = C - 1, i = D_k}$.
Therefore
\[
\mu(I) \leq \mu(I_0) + \sum_{k \geq 0} \mu (I_{k+1}/I_k)
\leq \mu(I_0) + \sum_{k \geq 0} \length (I_{k+1}/I_k)
= \mu(I_0) + \length (R/I_0).
\]
For the second assertion note that $\mu(I_0) \leq \eh(R) + \length (\lc_{\mf m}^0 (R))$ (for example, see \cite[Lemma 5.5]{HMQS}).
\end{proof}

We next prove a local Bertini-type result, this should be well-known to experts and the case of $s = 0$ follows from \cite[Theorem]{Hochster}.
We thank Bernd Ulrich for suggesting the argument.

\begin{lemma}
\label{lemma: Ulrich}
Let $(R,\m)$ be a Noetherian local ring which satisfies Serre's condition $(R_s)$ and has dimension at least $s+2$. If $\m=(x_1,\dots,x_n)$, then $R(t_1,\dots,t_n)/(t_1x_1+\cdots+t_nx_n)$ still satisfies $(R_s)$.
\end{lemma}
\begin{proof}
Let $P$ be a height $s+1$ prime in $S = R[t_1,\ldots, t_n]$ that contains $z=t_1x_1+\cdots+t_nx_n$. It is enough to show that $S_P/zS_P$ is regular.

Let $Q = P \cap R$. We first claim that $\hght(Q) \leq s$. For if $\hght(Q)= s+1$, then we must have
$P = Q[t_1, \ldots, t_n]$, but then $P$ cannot contain $z$ because $Q \neq \mf m$ (since $\hght(\m)=\dim(R)\geq s+2>\hght(Q)$), which is a contradiction. Thus, without loss of generality,
we assume that $x_1 \notin Q$, so $R_Q[t_1, \ldots, t_n]/(z) \cong R_Q[t_2, \ldots, t_n]$ is regular because $R_Q$ is regular. Therefore, $(S/zS)_P \cong (R_Q[t_2, \ldots, t_n])_P$ is also regular.
\end{proof}

We will need the following version of Lech's inequality, which is proved in \cite{HMQS} using Hanes' work on Hilbert--Kunz multiplicity \cite[Theorem 2.4]{Hanes} and reduction mod $p>0$.

\begin{theorem}[{\cite[Proposition 5.7]{HMQS}}]
\label{theorem: Hanes}
Let $d \geq 2$ and $N \geq d$ be two positive integers. Then there exists a constant $c = c(N, d) \in (0,1)$ such that for any equal characteristic Noetherian local ring $(R, \m)$ of dimension $d$ and any $\m$-primary ideal $I$ with $\mu(I)\leq N$, we have
\[
\eh(I) \leq d! c \eh(R) \length (R/I).
\]
In fact, one can take $c=(1-\frac{1}{N^{1/(d-1)}})^{d-1}$.
\end{theorem}

We now prove our main technical result.

\begin{theorem}\label{theorem: main technical dim 2}
Let $(R,\m)$ be a two-dimensional Noetherian complete local ring which satisfies Serre's condition $(R_0)$. If $\eh(R)>1$, then there exists $\epsilon>0$ such that $\eh(I)\leq 2(\eh(R)-\epsilon)\length(R/I)$ for all $\m$-primary ideals $I$.
\end{theorem}
\begin{proof}
Let $P_1,\dots,P_n$ be the minimal primes of $R$ such that $\dim(R/P_i)=2$. Since $R$ is $(R_0)$, we know that $0$ has a primary decomposition $$0=P_1\cap P_2\cap \cdots \cap P_n \cap P_{n+1}\cap \cdots \cap P_m\cap Q_1\cap\cdots \cap Q_k$$ where $P_{n+1},\dots, P_m$ are (possibly) minimal primes of $R$ whose dimensions are less than $2$ and $Q_1,\dots,Q_k$ are (possibly) embedded components. If we replace $R$ by $\widetilde{R}=R/(P_1\cap\cdots\cap P_n)$, then it follows by the additivity formula for multiplicities that for all $\m$-primary ideals $I\subseteq R$, we have $\eh(I, R)=\eh(I, \widetilde{R})$ while $\length(R/I)\geq\length(\widetilde{R}/I\widetilde{R})$. It follows that
$$\frac{\eh(I)}{2\cdot\length(R/I)}\leq \frac{\eh(I\widetilde{R})}{2\cdot \length(\widetilde{R}/I\widetilde{R})}.$$
Therefore to prove the result for $R$, it is enough to establish it for $\widetilde{R}$ (note that $\eh(R)=\eh(\widetilde{R})$). Thus we may replace $R$ by $\widetilde{R}$ to assume that $R$ is reduced and equidimensional.

Let $x_1,\dots,x_n$ be a generating set of $\m$. Note that by Lemma~\ref{lemma: Ulrich}, $R(t_1,\dots,t_n)/(t_1x_1+\cdots + t_nx_n)$ satisfies $(R_0)$.
Since it is excellent, its completion still satisfies $(R_0)$. Note that
the depth of the completion of $R(t_1,\dots,t_n)$ is at least one (since this is true for $R$).
Therefore, after replacing $R$ by the completion of $R(t_1,\dots,t_n)$, we may assume that there exists a nonzerodivisor $z\in \m$ such that $z$ is a part of a minimal reduction of $\m$ and $S = R/zR$ is $(R_0)$, and that the residue field of $R$ is infinite.
Note that $\eh(S)=\eh(R)$ since $z$ is part of a minimal reduction of $\m$. By \cite[Proposition 4.10]{HMQS}, there exists $C$ such that for all ideals $J\subseteq S$ with $\length(S/J)>C$, we have that $\eh(J)\leq \frac{3}{2} \cdot \length(S/J)$. Let us fix this $C$.

We first consider an arbitrary $\m$-primary ideal $I\subseteq R$ such that $\length(S/IS)\leq C$. We take the associated graded ring of $R$ with respect to the ideal $(z)$, and we use $\initial_z(I)$ to denote the initial ideal of $I$ in $\gr_z (R)$. Since $z$ is a nonzerodivisor, we have $\gr_{z} (R) \cong S[T]$. It follows that
$$\initial_z(I)=I_0+I_1T+\cdots +I_{N-1}T^{N-1}+T^N$$
where $I_0\subseteq I_1\subseteq \cdots\subseteq I_{N-1}$ are $\m_S$-primary ideals in $S$. Note that our assumption on $I$ says that $\length(S/I_0)\leq C$, which is a constant that does not depend on $I$, $N$, or any of the $I_i$. Since $\eh(\initial_z(I))\geq \eh(I)$ and $\length(R/I)=\length(S[T]/\initial_z(I))$, we have
\begin{equation}
\label{equation 1}
\frac{\eh(I)}{2\cdot \length(R/I)} \leq \frac{\eh(\initial_z(I))}{2\cdot \length(S[T]/\initial_z(I))}.
\end{equation}

We next take the associated graded ring of $S[T]$ with respect to $(\mf m, T)$.
By Lemma~\ref{lemma: double graded} the initial ideal of $\initial_z(I)$ is
\[
J := \initial_{\mf m} (I_0) + \initial_{\mf m} (I_1)T + \cdots +  \initial_{\mf m} (I_{N-1})T^{N-1} + T^N \subseteq \gr_{\m}(S)[T].
\]
Note that $\eh(J)\geq \eh(\initial_z(I))$ and $\length(S[T]/\initial_z(I))=\length((\gr_{\m}S)[T]/J)$, thus we have
\begin{equation}
\label{equation 2}
\frac{\eh(\initial_z(I))}{2\cdot \length(S[T]/\initial_z(I))} \leq \frac{\eh(J)}{2 \cdot \length (R/J)}.
\end{equation}
By Lemma~\ref{lemma: bound generators}, $\mu(J)\leq C + \eh(S) + \length (\lc_{ \initial_{\mf m}(\m)}^0 (\gr_{\mf m} (S))$.
Since $\gr_{\mf m} (S)[T]$ contains a field, $S/\m S$, and has dimension two,
by Theorem~\ref{theorem: Hanes}, there exists a constant $0<\epsilon\ll1$
(which depends on $C$, $R$ and $S$, but not on $I$!)
such that
\begin{equation}\label{equation 3}
\frac{\eh(J)}{2 \cdot \length (R/J)} \leq (1 - \epsilon) \eh(\gr_{\mf m} (S)[T]) = (1 - \epsilon) \eh(R).
\end{equation}
Putting (\ref{equation 1}), (\ref{equation 2}), (\ref{equation 3}) together, we have proved the theorem for all $I$ such that $\length(S/IS)\leq C$.
We can further shrink $\epsilon$ to guarantee that $\frac{3}{2} < (1 - \epsilon) \eh(R)$ since $\eh(R)>1$.

Finally, we use induction on $\length(R/I)$ to show that $\epsilon$ works for all $\m$-primary ideal $I\subseteq R$. We may assume that $\length(S/IS)>C$, then by \cite[Lemma 5.1]{HMQS} and the second paragraph of the proof, we have
$$\frac{\eh(I)}{2\cdot\length(R/I)}\leq \max\left\{\frac{\eh(I:z)}{2\cdot \length(R/(I:z))}, \frac{\eh(IS)}{\length(S/IS)}\right\} \leq
\max \left\{(1 - \epsilon) \eh(R), \frac{3}{2} \right\} = (1 - \epsilon) \eh(R), $$
where for the second inequality we are using induction on the colength.
\end{proof}

We next deduce the higher dimensional case from the two-dimensional case via induction on dimension, this is similar to the strategy in \cite[Theorem 5.8]{HMQS}, the only difference is that here we use Lemma~\ref{lemma: Ulrich} instead of Flenner's result \cite[Lemma 5.4]{HMQS} (in equal characteristic).

\begin{corollary}
\label{cor: main  R0}
Let $(R, \m)$ be a Noetherian local ring of dimension $d$ such that $\widehat{R}$ satisfies Serre's condition $(R_0)$. If $d\geq 2$ and $\eh(R)>1$, then there exists $\epsilon>0$ such that $\eh(I)\leq d!(\eh(R)-\epsilon)\length(R/I)$ for all $\m$-primary ideals $I$.
\end{corollary}
\begin{proof}
We may assume that $R$ is complete. 
We use induction on $d \geq 2$. Theorem~\ref{theorem: main technical dim 2}
provides the base case. Suppose $d\geq 3$ and that $x_1, \ldots, x_n$ is a generating set for $\mf m$. We can replace $R$ by $R(t_1, \ldots, t_n)$. Then we consider $R' = R(t_1, \ldots, t_n)/(t_1x_1 + \cdots + t_nx_n)$. By Lemma~\ref{lemma: Ulrich}, $R'$ (and hence $\widehat{R'}$) still satisfies $(R_0)$, $\dim(R') = d- 1$, and $\eh(R') = \eh(R)$. By induction the assertion holds for $R'$. That is, there exists $\epsilon$ such that $\eh(J) \leq (d-1)!(\eh(R') - \epsilon) \length (R'/J)$ for any $\m$-primary ideal $J\subseteq R'$. We use induction on $\length (R/I)$ to show that the same $\epsilon$ works for $R$ (the initial case $I=\m$ is obvious). By \cite[Lemma 5.1]{HMQS} we have
\[
\frac{\eh(I)}{d!\length (R/I)}
\leq \max \left\{\frac{\eh(I:z)}{d!\length (R/(I:z))}, \frac{\eh(IR')}{(d-1)!\length (R'/IR')} \right\}
\leq \eh(R') - \epsilon=\eh(R)-\epsilon.  \]
where $z=t_1x_1 + \cdots + t_nx_n$. This completes the proof.
\end{proof}

Here is our uniform Lech's inequality, now valid in all characteristics.

\begin{theorem}[Uniform Lech's inequality]
\label{theorem: uniform Lech}
Let $(R, \mf m)$ be a Noetherian local ring of dimension $d \geq 2$. Suppose $\eh({\widehat{R}}_{\red})>1$. Then there exists $\varepsilon > 0$ such that for any $\m$-primary ideal $I$, we have
$$
\eh(I) \leq d!(\eh(R) - \varepsilon) \length (R/I).
$$
\end{theorem}
\begin{proof}
This follows by the same argument as in \cite[Proof of Corollary 5.9]{HMQS}, we just replace the citation of \cite[Theorem 5.8]{HMQS} by Corollary~\ref{cor: main  R0}  above.
\end{proof}

Now we can prove Conjecture~\ref{conj: asymptoticLech}.

\begin{corollary}
\label{cor: asymptotic Lech}
Let $(R,\m)$ be a Noetherian local ring of dimension $d\geq 1$. Then we have $\eh(\widehat{R}_{\red}) > 1$ if and only if
\[
\lim_{N\to\infty} \sup_{\substack{\sqrt{I}=\m \\ \length(R/I)> N}} \left\{\frac{\eh(I)}{d!\length(R/I)} \right\}<\eh(R).
\]
\end{corollary}
\begin{proof}
The ``if" direction was proved in \cite[Proposition 5.3]{HMQS}. For the ``only if" direction, the one-dimensional case was proved in \cite[Proposition 5.11]{HMQS}, and when $d\geq 2$, the result follows immediately from Theorem~\ref{theorem: uniform Lech}.
\end{proof}

\subsection{Uniform Lech's inequality for ideals with fixed number of generators}
We conjecture that Theorem~\ref{theorem: Hanes} holds without the assumption on characteristic. We are able to show this for \emph{integrally closed} (more generally $\mf m$-full) ideals, which is sufficient for giving a different proof of the Uniform Lech's inequality, see Remark~\ref{rmk: different proof}.
The proof is different  in dimension two and for higher dimensions.
We start with the former, for which we will need the following corollary of Lemma~\ref{lemma: bound generators}.

\begin{corollary}\label{cor: quadratic bound}
Let $(R, \mf m)$ be a Noetherian local ring and $I$ be an $\m$-full $\m$-primary ideal.
Let $\m=(x_1,\dots,x_n)$ and define
$S = R(t_1,\dots,t_n)/\lc^0_{\m} (R)R(t_1,\dots,t_n)$
with the general linear form $z = t_1x_1+\cdots+t_nx_n$.
Then $\initial_z(IS)$, the initial ideal of $IS$ in $\gr_z (S)$, can be generated by at most
$\mu(I)$ homogeneous elements.
\end{corollary}
\begin{proof}
Since $S$ has positive depth, $z$ is a nonzerodivisor on $S$.
By Lemma~\ref{lemma: bound generators}, we know that $\initial_z(IS)$ can be generated by at most $\length (S/(I, z)) + \mu(I(S /zS))$
(homogeneous) elements. Both the minimal number of generators and the colength do not increase when passing to a quotient ring. Hence if we let $\widetilde{R} = R(t_1,\dots,t_n)$
then the above bound is no greater than
$
\length (\widetilde{R}/(I, z)\widetilde{R}) + \mu(I(\widetilde{R} /z\widetilde{R})).
$ Because $I$ is $\m$-full, we know that $\mu(I)=\mu(I(\widetilde{R} /z\widetilde{R})) + \length (\widetilde{R} /(I, z)\widetilde{R})$ by Remark~\ref{remark: property m-full}.
\end{proof}

\begin{theorem}\label{theorem: dim 2 generators}
For any Noetherian local ring $(R, \m)$ of dimension two and any $\m$-primary $\m$-full ideal $I$ which can be generated by $N$ elements we have
\[
\eh(I) \leq 2 \left (1 - \frac 1{2N-2}\right) \eh(R) \length (R/I).
\]
\end{theorem}
\begin{proof}
We use the notation of Corollary~\ref{cor: quadratic bound}.
Observe that passing from $R$ to $S$ does not affect multiplicity and does not increase the colength. Let $J:=\initial_z(IS)$ be the initial ideal of $IS$ in $\gr_z(S)\cong (S/zS)[T]$. We know that
\begin{equation}
\label{equation in two-dim Hanes}
\frac{\eh(I)}{2\cdot\length(R/I)}\leq \frac{\eh(IS)}{2\cdot\length(S/IS)}\leq \frac{\eh(J)}{2\cdot\length(\gr_z(S)/J)}.
\end{equation}

We next write $J = J_0 + J_1T + \cdots + J_{K - 1}T^{K-1} + T^K$ as a $T$-homogenous ideal of $(S/zS)[T]$ with $J_{K-1} \neq R$. By Corollary~\ref{cor: quadratic bound}, we know that $J$ can be generated by at most $N$ homogeneous elements. We define the sequence $\{n_k\}$ that labels the distinct $J_i$ by setting $n_0 = 0$ and
$n_{k+1} =  \min \{n < K \mid J_n \neq J_{n_k}\}$.
We now choose appropriately $\leq N - 1$ generators $\{a_{i, j}T^j\}$ of $J$, so that
$a_{1, 0}, a_{2, 0}, \ldots $ generate $J_0=J_{n_0}$ and
$J_{n_k}$ is generated by $a_{i, j}$ with $j \leq k$ (thus $J$ is generated by $T^k$ and $a_{i, j}T^{n_j}$).
By adjoining one more generator to each $J_{n_k}$ if necessary,
we may assume that the chosen generating set contain a minimal reduction
of each $J_{n_k}$ as one of $a_{i, k}$ (note that each $J_{n_k}$ is an ideal in the one-dimensional ring $S/zS$).
The total number of adjoined generators is at most $N - 2$, because there are at most $N - 1$ ideals $J_{n_k}$s and
we do not need to adjoin a new generator to $J_{n_0}$ (since we can let $a_{1,0}$ be a minimal reduction of $J_0$).

Inspired by positive characteristic methods, for each positive integer $q$ we define
$J^{[q]}$ as the ideal generated by $\{a_{i, j}^qT^{jq}\}$ and define $J_{i}^{[q]}$ accordingly. Note that $J^{[q]}$ and each $J_i^{[q]}$ in principle might depend on the chosen generating set $\{a_{i,j}T^j\}$ but this will not be a problem.
By definition, we have
$$J^{[q]}= J_0^{[q]} + J_0^{[q]} T + \cdots J_0^{[q]} T^{q - 1} + J_1^{[q]}T^q + \cdots+  J_{1}^{[q]}T^{2q-1} + J_2^{[q]}T^{2q}+ \cdots .$$ It follows that $\length ((S/zS)[T]/J^{[q]}) = q \sum_{i=0}^{K-1} \length (S/(z, J_i^{[q]}))$.
Note that since $\dim(S/zS) = 1$ and our selected list of generators
contains a minimal reduction $a_i$ of every $J_i$, we have that
\[\eh (J_i) =
\lim_{q \to \infty} \frac{\length (S/(z, a_i^q))}{q} \geq
\lim_{q \to \infty} \frac{\length (S/(z, J_i^{[q]}))}{q}
\geq \lim_{q \to \infty} \frac{\length (S/(z, J_i^{q}))}{q}
= \eh (J_i).
\]
Therefore we have
$$\lim_{q \to \infty}  \frac{\length((S/zS)[T]/J^{[q]})}{q^2}=\sum_{i=0}^{K-1}\frac{\length (S/(z, J_i^{[q]}))}{q} =\sum_{i=0}^{K-1}\eh(J_i).$$
Applying Lech's inequality (Theorem~\ref{theorem: Lech}) to each $J_i\subseteq S/zS$, we then have
\begin{equation}\label{frq equation}
\lim_{q \to \infty} \frac{\length ((S/zS)[T]/J^{[q]})}{q^2} \leq \eh (S/zS) \sum_{i=1}^{K-1} \length (S/(z, J_i)) = \eh (S/zS) \length ((S/zS)[T]/J).
\end{equation}

At this point, we follow the argument in \cite[Theorem 2.4]{Hanes}.
First, since $J^{[q]}$ is clearly contained in $J^q$ and is generated by at most $2N-2$ elements, for any integer $s$ we can surject $2(N-1)$ copies of $(S/zS)[T]/J^s$  onto $J^{[q]}/(J^{[q]} \cap J^{q + s})$. Thus
$$\length ((S/zS)[T]/J^{[q]}) \geq \length ((S/zS)[T]/J^{q + s}) - 2(N-1) \length ((S/zS)[T]/J^s).$$
As in \cite[Theorem 2.4]{Hanes}, setting $s = \lceil q/(2N-3) \rceil$ will yield
\[
\lim_{q \to \infty} \frac{\length ((S/zS)[T]/J^{[q]})}{q^2}
\geq  \frac{\eh(J)}{2} \left (\left(1 + \frac{1}{2N-3} \right)^2 - \frac{2N-2}{2N-3} \right )
= \frac{\eh(J)} 2 \left(1 + \frac{1}{2N-3} \right).
\]
Combining with (\ref{frq equation}) we obtain that
\[
\frac{\eh(J)}{2\cdot \length ((S/zS)[T]/J)} \leq \eh(S/zS) \left(1 + \frac{1}{2N-3} \right)^{-1} = \eh(R) \left(1 + \frac{1}{2N-3} \right)^{-1}.
\]
This together with (\ref{equation in two-dim Hanes}) completes the proof.
\end{proof}

\begin{corollary}\label{cor: no Hanes bound}
Let $(R, \m)$ be a two-dimensional Noetherian local ring.
Let $\m=(x_1,\dots,x_n)$ and define $R' = R(t_1,\dots,t_n)/(t_1x_1+\cdots+t_nx_n)$.
Then for any positive integer $N$ there exists $\varepsilon > 0$ such that
for any $\mf m$-primary ideal $I$ with $\length (R'/IR') \leq N$, we have
\[
\eh(I)\leq 2(1 - \varepsilon)\eh(R)\length(R/I).
\]
\end{corollary}
\begin{proof}
We may assume $I$ is integrally closed since replacing $I$ by its integral closure $\overline{I}$ will not affect $\eh(I)$ and will not increase $\length(R/I)$.
By Remark~\ref{remark: property m-full} and \cite[Lemma 5.5]{HMQS},
$\mu(I) \leq N + \mu(IR')$ is then bounded by a constant independent of $I$, so we may apply
Theorem~\ref{theorem: dim 2 generators}.
\end{proof}

\begin{remark}\label{rmk: different proof}
One can give an alternative proof of Theorem~\ref{theorem: main technical dim 2} (and thus
the uniform Lech's inequality Theorem~\ref{theorem: uniform Lech}) via Corollary~\ref{cor: no Hanes bound}: in fact, Corollary~\ref{cor: no Hanes bound} handles exactly the case $\length(S/IS)\leq C$ in the proof of Theorem~\ref{theorem: main technical dim 2}.
This alternative approach avoids the use of Theorem~\ref{theorem: Hanes} (i.e., \cite[Proposition~5.7]{HMQS}), which benefits in equal characteristic $0$
as it avoids the reduction mod $p$ argument needed to prove Theorem~\ref{theorem: Hanes}.
\end{remark}

Finally, we treat the higher dimensional case, the proof turns out to be easier, but, unlike Theorem~\ref{theorem: dim 2 generators}, the bound is not sharp for the maximal ideal in a regular local ring. We need a couple of lemmas. The first one is due to Mumford.

\begin{lemma}[{\cite[Proof of Lemma 3.6]{Mumford}}]
\label{lemma: Mumford}
Let $(R,\m)$ be a Noetherian local ring of dimension $d$ and let $I=I_0+I_1T+I_2T^2+\cdots+I_{N-1}T^{N-1}+T^N \subseteq R[T]$ be a $T$-homogeneous ideal of finite colength. Then we have
$$\frac{\eh(I)}{(d+1)!\length(R[T]/I)} \leq \frac{\sum_{i=0}^{N-1}\eh(I_i)}{d!\sum_{i=0}^{N-1}\length(R/I_i)}\leq \max_i\left\{\frac{\eh(I_i)}{d!\length(R/I_i)}\right\}.$$
\end{lemma}

\begin{lemma}\label{lemma: Lech colength}
For any Noetherian local ring $(R, \m)$ of dimension $d \geq 2$ and any $\m$-primary ideal $I$ of colength at most $N$ we have
\[
\eh(I) \leq d! \left(1 - \frac{1}{d! N} \right ) \eh(R) \length (R/I).
\]
\end{lemma}
\begin{proof}
The proof of Lech's inequality via Noether's normalization given in \cite{HSV}
shows that it is enough to prove the statement in an equicharacteristic regular local ring of dimension $d$. Since it was shown by Lech \cite[Page 74, after (4.1)]{LechMultiplicity} that in dimension at least two, we always have strict inequality in Theorem~\ref{theorem: Lech} (so that each rational number appeared on the left hand side below must be strictly less than $1$), it follows that
\[
\max_{\length(R/I)\leq N}
\left \{\frac{\eh(I)}{d! \length (R/I)}\right\} \leq 1 - \frac{1}{d! N}.
\qedhere \]
\end{proof}

\begin{theorem}\label{theorem: dim d generators}
Let $d > 2$ and $N \geq d$ be two positive integers.
Then there exists a constant $c = c(N, d) \in (0, 1)$ such that for any Noetherian local ring $(R, \m)$ of dimension $d$ and any $\m$-primary $\m$-full ideal $I$
which can be generated by $N$ elements we have
\[
\eh(I) \leq d! c \eh(R) \length (R/I).
\]
\end{theorem}
\begin{proof}
Let us write $\m=(x_1,\dots,x_n)$ and define
$\widetilde{R} = R(t_1,\dots,t_n)$
with the general linear form $z = t_1x_1+\cdots+t_nx_n$.
By Remark~\ref{remark: property m-full}, for any $\m$-primary $\m$-full ideal $I$ we have
\[
\mu(I(\widetilde{R} /z\widetilde{R})) + \length (\widetilde{R} /(I, z)\widetilde{R}) = \mu(I) \leq N.
\]
We apply the same reduction as in the first paragraph of the proof of Theorem~\ref{theorem: dim 2 generators}: we can pass from $R$ to $S := (R/\lc_{\m}^0 (R)) \otimes \widetilde{R}$ because this will not affect multiplicity and will not increase the colength, now $z$ is a nonzerodivisor on $S$ and we can further pass from $S$ to $\gr_z (S) \cong (S/zS)[T]$ and note that if $c$ works for all $T$-homogeneous ideals in $\gr_z (S)$, then it will work for $S$.

We now write $\initial_z(IS)$, the initial ideal of $IS$ in $\gr_z(S)\cong (S/zS)[T]$, as $I_0+I_1T+I_2T^2+\cdots+I_{K-1}T^{K-1}+T^K$. By Lemma~\ref{lemma: Mumford}, it is enough to show that there exists $c>0$ such that $\eh(I_i)\leq (d-1)!(1-c)\eh(S/zS)\length(S/(z,I_i))$. But now we have $\length(S/(z,I_i))\leq N$ for each $I_i$ and $\dim(S/zS)=d-1\geq 2$. Thus the assertion follows from Lemma~\ref{lemma: Lech colength} applied to $S/zS$ (and we can actually take $c=1 - \frac{1}{(d-1)!N}$).
\end{proof}

\begin{corollary}
\label{cor: weak Hanes}
Let $d \geq 2$ and $N \geq d$ be two positive integers. Then there exists a constant $c = c(N, d) \in (0, 1)$ such that for any Noetherian local ring $(R, \m)$ of dimension $d$ and any $\m$-primary integrally closed ideal $I$ which can be generated by $N$ elements we have
\[
\eh(I) \leq d! c \eh(R) \length (R/I).
\]
\end{corollary}
\begin{proof}
Since we are in dimension at least two, any $\m$-primary integrally closed ideal is $\m$-full, see Remark~\ref{remark: property m-full}. The conclusion follows from Theorem~\ref{theorem: dim 2 generators} when $d=2$ and Theorem~\ref{theorem: dim d generators} when $d>2$.
\end{proof}

As we mentioned in the introduction, we expect Corollary~\ref{cor: weak Hanes} holds without assuming $I$ is integrally closed (and recall that this is true in equal characteristic by Theorem~\ref{theorem: Hanes}).

\begin{conjecture}
\label{conj: Hanes}
Let $d \geq 2$ and $N \geq d$ be two positive integers. Then there exists a constant $c = c(N, d) \in (0, 1)$ such that for any Noetherian local ring $(R, \m)$ of dimension $d$ and any $\m$-primary ideal $I$ which can be generated by $N$ elements we have
\[
\eh(I) \leq d! c \eh(R) \length (R/I).
\]
\end{conjecture}

\bibliographystyle{plain}
\bibliography{refs}

\end{document}